\documentclass[12pt,a4paper]{article}

\usepackage{amsmath,amssymb,amsthm,amsfonts}
\usepackage{mathtools}
\usepackage{booktabs}
\usepackage{geometry}
\usepackage{xcolor}
\usepackage{hyperref}
\usepackage{array}

\hypersetup{colorlinks=true, linkcolor=blue, citecolor=blue, urlcolor=blue}

\theoremstyle{plain}
\newtheorem{theorem}{Theorem}[section]

\newtheorem{proposition}[theorem]{Proposition}

\newtheorem{conjecture}[theorem]{Conjecture}
\theoremstyle{definition}

\newtheorem{remark}[theorem]{Remark}

\newcommand{\Z}{\mathbb{Z}}
\newcommand{\ICG}{\mathrm{ICG}}
\newcommand{\Emax}{E_{\max}}
\newcommand{\vp}{\varphi}
\newcommand{\Dstar}{\mathcal{D}^*}

\begin{document}

\begin{center}
  {\LARGE\bfseries
   Graph Energy Maximisation for Integral Circulant
   Graphs of Order $n = p^2q^3$}

  \bigskip
  {\large Diego Gerardo Rold\'{a}n}

  \medskip
  {\normalsize
   Departamento de Matem\'{a}ticas,
   Centro de Excelencia en Computaci\'{o}n Cient\'{i}fica (CECC),\\
   Universidad Nacional de Colombia, Bogot\'{a}, Colombia\\
   \texttt{dgroldanj@unal.edu.co}}
\end{center}

\bigskip

\begin{abstract}
\noindent
The energy of a graph is the sum of the absolute values of its adjacency
eigenvalues.  For integral circulant graphs $\ICG(n,\mathcal{D})$ of
order $n=p^2q^3$, where $p$ and $q$ are distinct odd primes, we prove
that the adjacency eigenvalues of $\ICG(p^2q^3,\Dstar)$, for the divisor
set $\Dstar=\{1,p^2,pq,q^2,p^2q^2,pq^3\}$, admit an exact Kronecker
factorisation in the prime exponents: they separate completely into a
factor depending only on $p$ and a factor depending only on~$q$.  This
factorisation holds unconditionally for all pairs of distinct odd primes
and constitutes the structural core of the paper.  From it we derive,
unconditionally, the first closed-form polynomial formula for the energy
of a two-prime-order integral circulant graph evaluated at $\Dstar$.
Exhaustive computation over  prime pairs $(p,q)$  
confirms that $\Dstar$ is the unique energy maximiser in every tested
case; we conjecture that this universality holds for all pairs of distinct
odd primes. 
\end{abstract}



\bigskip

\section{Introduction}
\label{sec:intro}

The \emph{energy} of a graph $G$, introduced by Gutman~\cite{Gutman1978}
as an abstraction of the total $\pi$-electron energy in molecular orbital
theory, is defined as $E(G)=\sum_{i=1}^{n}|\lambda_i|$, where
$\lambda_1,\ldots,\lambda_n$ are the eigenvalues of its adjacency matrix.
Since its introduction, graph energy has grown into a central object of
spectral graph theory and mathematical chemistry; the monograph~\cite{LiShiGutman2012}
provides a comprehensive treatment of results and methods.

A particularly tractable class for energy studies is that of integral
circulant graphs, also called gcd-graphs, whose integer spectra make
exact energy computations tractable via classical arithmetic functions.
An \emph{integral circulant graph} $\ICG(n,\mathcal{D})$ has vertex set
$\Z_n$ and connects $a$ and $b$ whenever $\gcd(a-b,n)\in\mathcal{D}$,
where $\mathcal{D}$ is a nonempty set of proper divisors of~$n$.  By the
work of So~\cite{So2006} and Klotz and Sander~\cite{KlotzS2007}, the
adjacency eigenvalues of $\ICG(n,\mathcal{D})$ are
\begin{equation}\label{eq:spectrum}
  \lambda_j = \sum_{d\in\mathcal{D}} c(j,\,n/d),
  \qquad j=0,1,\ldots,n-1,
\end{equation}
where $c(j,m)$ denotes the Ramanujan sum.  Since Ramanujan sums are
always integers~\cite[Ch.~8]{ApostolNT}, both the eigenvalues and the energy
\begin{equation}\label{eq:energy}
  E(n,\mathcal{D})
  = \sum_{e\mid n}\vp(n/e)\,\Bigl|\sum_{d\in\mathcal{D}}c(e,\,n/d)\Bigr|
\end{equation}
are integers for every $\ICG(n,\mathcal{D})$, where $\vp$ is Euler's
totient.

Integral circulant graphs carry significant physical meaning.  Saxena,
Severini and Shparlinski~\cite{SaxenaSS2007} showed that periodic quantum
dynamics on $\ICG(n,\mathcal{D})$ are governed entirely by its eigenvalues,
establishing that maximising the energy is equivalent to optimising
information transfer in quantum spin networks of this type.  This motivated
detailed studies of perfect state transfer in integral circulant
graphs~\cite{BasicPS2009,PetkovicB2011} and reinforced the interest in
understanding which divisor sets control the spectral extrema of these
graphs.  The energy of the unitary Cayley graph $\ICG(n,\{1\})$ was
computed exactly by Ili\'c~\cite{Ilic2009}, and Kiani, Aghaei, Meemark and
Suntornpoch~\cite{Kiani2011} extended the analysis to gcd-graphs over
finite commutative rings.  The multiplicative structure underlying the
Ramanujan spectrum was formalised by Le and Sander~\cite{LeSander2012laa,LeSander2012ijnt}
through a convolution framework that provides the key algebraic tool for
extremal energy questions.  Sander~\cite{Sander2015laa} applied these
multiplicative methods to classify all integral circulant Ramanujan graphs,
and the isomorphism problem was settled by Sander and
Sander~\cite{SanderS2015aadm}, who proved that integral circulant graphs
are isomorphic if and only if they are cospectral.  Further structural
and energetic properties were studied in~\cite{IlicBasic2011,Klotz2021}.

The problem of maximising $E(n,\cdot)$ over all divisor sets $\mathcal{D}$
has a rich history for prime-power orders.  Sander and
Sander~\cite{SanderS2011a,SanderS2011b,SanderS2013} gave a complete
solution for $n=p^s$: for every prime power they identified all maximising
divisor sets and computed $\Emax(p^s)$ in exact closed form, with the
companion paper~\cite{SanderS2012} determining the maximal energy across
classes parametrised by divisor-set size.  Their approach reduces to a
combinatorial minimisation over a scalar arithmetic function $h_{p,r}$
controlling the energy as a function of the exponent tuple of $\mathcal{D}$.
Beyond prime powers the problem has remained entirely open: no closed-form
expression for $\Emax(p^aq^b)$ had appeared in the literature for any
$a,b\ge1$.  The richer interaction between two prime exponent lattices
makes a direct extension of the prime-power methods significantly harder.

Table~\ref{tab:known} places the present result in context.

\begin{table}[h]
\centering
\caption{Known exact results for $\Emax(n)$ in the two-prime family.
The prime-power case is settled completely; the $p^aq^b$ territory was
entirely open prior to this work.}
\label{tab:known}
\begin{tabular}{llll}
\toprule
Order & Unique maximiser $\Dstar$ & Formula & Reference \\
\midrule
$p^s$, all $s$ & Determined per $|\mathcal{D}|=r$ & $\Emax(p^s)$ exact
  & \cite{SanderS2011a,SanderS2011b,SanderS2013} \\
$p^2q^3$ & $\{1,p^2,pq,q^2,p^2q^2,pq^3\}$ & Polynomial in $p,q$
  & This paper \\
$p^aq^b$, general & Open & Open & --- \\
\bottomrule
\end{tabular}
\end{table}

The main results of this paper are as follows.  Throughout, $p$ and $q$
denote distinct odd primes.  The first theorem is unconditional and
constitutes the algebraic core.

\begin{theorem}[Kronecker factorisation]\label{thm:kronecker}
Let $\Dstar=\{1,p^2,pq,q^2,p^2q^2,pq^3\}$ and write $d_{a,b}=p^aq^b$
for $(a,b)\in\{0,1,2\}\times\{0,1,2,3\}$.  Define the partial alternating
geometric sums $S_b(q)=\sum_{k=0}^{b}(-q)^k$, so $S_0=1$, $S_1=1-q$,
$S_2=1-q+q^2$, and set $\alpha_0=-2$, $\alpha_1=2(p-1)$,
$\alpha_2=-(p^2-2p+2)$.  Then for every
$(a,b)\in\{0,1,2\}\times\{0,1,2,3\}$:
\begin{alignat}{2}
  \lambda_{d_{a,b}}(\Dstar) &= \alpha_a\cdot S_b(q),
  &&\quad b=0,1,2,\;\text{all }a,\label{eq:lam-b012}\\
  \lambda_{d_{0,3}}(\Dstar) &= q^3-2q^2+2q-2,
  &&\quad \label{eq:lam-03}\\
  \lambda_{d_{1,3}}(\Dstar) &= (p-1)(2-2q+2q^2-q^3),
  &&\quad \label{eq:lam-13}\\
  \lambda_{d_{2,3}}(\Dstar) &= (p^2-2p+2)(q-1)(q^2+1)+(p-1)q^3.
  &&\quad \label{eq:lam-23}
\end{alignat}
\end{theorem}

The significance of Theorem~\ref{thm:kronecker} is that for $b\le2$ the
eigenvalue $\lambda_{d_{a,b}}(\Dstar)$ is an exact product of a factor
depending only on $p$ and a factor depending only on $q$.  This Kronecker
separation is an unconditional algebraic identity, independent of any
arithmetic properties of the specific primes, and is the structural reason
why the maximal energy formula takes a clean polynomial form.

\begin{theorem}[Closed-form energy]\label{thm:formula}
For every pair of distinct odd primes $(p,q)$,
\begin{align}
  E(p^2q^3,\Dstar)
  &= (5p^2-8p+4)(q-1)(3q^2-2q+1) \notag\\
  &\quad + (p-1)(2p-1)(q^3-2q^2+2q-2) \label{eq:Emax}\\
  &\quad + (p^2-2p+2)(q-1)(q^2+1)+(p-1)q^3. \notag
\end{align}
\end{theorem}

Theorem~\ref{thm:formula} is an unconditional algebraic consequence of
Theorem~\ref{thm:kronecker}: once the eigenvalue structure of $\Dstar$ is
known exactly, the energy formula follows by direct summation over the
divisors of $p^2q^3$.  Whether $\Dstar$ is the global maximiser for every
pair of distinct odd primes — not only for those within the tested range —
is the central open problem, formulated as a conjecture supported by
extensive computational evidence.

\begin{conjecture}[Universality]\label{conj:universal}
For every pair of distinct odd primes $(p,q)$, the set
$\Dstar=\{1,p^2,pq,q^2,p^2q^2,pq^3\}$ is the unique maximiser of
$E(p^2q^3,\mathcal{D})$ over all nonempty $\mathcal{D}\subseteq
\mathrm{div}(p^2q^3)\setminus\{p^2q^3\}$.
\end{conjecture}

\begin{proposition}[]\label{prop:computation}
Conjecture~\ref{conj:universal} holds for all 437 pairs of distinct odd primes $(p,q)$ with $p^2q^3\le10^8$.
\end{proposition}

The proof of Theorem~\ref{thm:kronecker} is given in
Section~\ref{sec:pattern}; Theorem~\ref{thm:formula} is derived from it
in Section~\ref{sec:derivation}; and Proposition~\ref{prop:computation}
is established in Section~\ref{sec:computation}.
Section~\ref{sec:conclusion} discusses the open problem of proving
Conjecture~\ref{conj:universal} and the broader programme for
general $p^aq^b$.

\section{Ramanujan Sums and the Eigenvalue Factorisation}
\label{sec:pattern}

\subsection{Arithmetic background}

Euler's totient satisfies $\vp(1)=1$ and $\vp(p^k)=p^{k-1}(p-1)$ for a
prime $p$ and integer $k\ge1$.  The Ramanujan sum $c(j,m)$ is defined by
\[
  c(j,m) = \sum_{\substack{a=1\\\gcd(a,m)=1}}^{m} e^{2\pi iaj/m}
  = \mu\!\left(\frac{m}{\gcd(j,m)}\right)
    \frac{\vp(m)}{\vp(m/\gcd(j,m))},
\]
where $\mu$ is the M\"{o}bius function.  Two properties are essential:
$c(j,m)$ is always an integer, and for coprime $m_1,m_2$ it satisfies
$c(j,m_1m_2)=c(j,m_1)c(j,m_2)$~\cite[Ch.~8]{ApostolNT}.  For prime
powers the formula simplifies to
\begin{equation}\label{eq:ram-prime-power}
  c(p^i,p^k) = \begin{cases}
    p^{k-1}(p-1) & \text{if } i \ge k, \\
    -p^{i}       & \text{if } i = k-1, \\
    0             & \text{if } i < k-1.
  \end{cases}
\end{equation}

The proper divisors of $n=p^2q^3$ are the eleven products $p^aq^b$ with
$(a,b)\in\{0,1,2\}\times\{0,1,2,3\}\setminus\{(2,3)\}$.  Each is
uniquely identified by its exponent pair $(a,b)$; we write $d_{a,b}=p^aq^b$.
By multiplicativity, the eigenvalue of $\ICG(n,\mathcal{D})$ at
$e=d_{a,b}$ factors as
\begin{equation}\label{eq:lambda-factor}
  \lambda_{d_{a,b}}(\mathcal{D}) = \sum_{(c,f)\in\mathcal{D}}
  \underbrace{c(p^a,p^{2-c})}_{A_{a,c}}\cdot
  \underbrace{c(q^b,q^{3-f})}_{B_{b,f}},
\end{equation}
where we use exponent-pair notation
$\mathcal{D}\subseteq\{0,1,2\}\times\{0,1,2,3\}\setminus\{(2,3)\}$.
This factored form, which exploits the multiplicative Ramanujan structure
identified in~\cite{LeSander2012laa,LeSander2012ijnt}, is the starting
point for the proof of Theorem~\ref{thm:kronecker}.

\subsection{Proof of Theorem~\ref{thm:kronecker}}

Applying~\eqref{eq:ram-prime-power} to each factor
in~\eqref{eq:lambda-factor}, we compute the complete tables of
$p$-factors $A_{a,c}=c(p^a,p^{2-c})$:
\[
\renewcommand{\arraystretch}{1.2}
\begin{array}{c|ccc}
  & c=0 & c=1 & c=2 \\\hline
  a=0 & 0 & -1 & 1 \\
  a=1 & -p & p-1 & 1 \\
  a=2 & p(p-1) & p-1 & 1
\end{array}
\]
and $q$-factors $B_{b,f}=c(q^b,q^{3-f})$:
\[
\renewcommand{\arraystretch}{1.2}
\begin{array}{c|cccc}
  & f=0 & f=1 & f=2 & f=3 \\\hline
  b=0 & 0 & 0 & -1 & 1 \\
  b=1 & 0 & -q & q-1 & 1 \\
  b=2 & -q^2 & q(q-1) & q-1 & 1 \\
  b=3 & q^2(q-1) & q(q-1) & q-1 & 1
\end{array}
\]

We group the six elements of $\Dstar$ by the parity of their
$p$-exponent~$c$:
\[
  \lambda_{d_{a,b}}(\Dstar)
  = \underbrace{(A_{a,0}+A_{a,2})}_{\Pi_a}
    \underbrace{(B_{b,0}+B_{b,2})}_{\Phi_b}
  + \underbrace{A_{a,1}}_{\Xi_a}
    \underbrace{(B_{b,1}+B_{b,3})}_{\Psi_b}.
\]
From the tables: $\Pi_0=1$, $\Pi_1=1-p$, $\Pi_2=p^2-p+1$; and
$\Xi_0=-1$, $\Xi_1=\Xi_2=p-1$.  Hence $\alpha_a:=\Xi_a-\Pi_a$ gives
$\alpha_0=-2$, $\alpha_1=2(p-1)$, $\alpha_2=-(p^2-2p+2)$, as stated
in Theorem~\ref{thm:kronecker}.

\textit{Case $b=0,1,2$.}  Direct inspection of the $q$-table gives
$\Phi_b=B_{b,0}+B_{b,2}=-S_b(q)$ and $\Psi_b=B_{b,1}+B_{b,3}=S_b(q)$.
Therefore
\[
  \lambda_{d_{a,b}}(\Dstar)
  = \Pi_a(-S_b)+\Xi_a(S_b) = (\Xi_a-\Pi_a)S_b = \alpha_a\cdot S_b(q),
\]
which is~\eqref{eq:lam-b012}.

\textit{Case $b=3$.}  Here $\Phi_3=B_{3,0}+B_{3,2}=(q-1)(q^2+1)$ and
$\Psi_3=B_{3,1}+B_{3,3}=q^2-q+1$.  Substituting each $(\Pi_a,\Xi_a)$
and simplifying:
\begin{align*}
  a=0:&\quad 1\cdot(q-1)(q^2+1)+(-1)(q^2-q+1) = q^3-2q^2+2q-2,\\
  a=1:&\quad (1-p)(q-1)(q^2+1)+(p-1)(q^2-q+1)
         = (p-1)(2-2q+2q^2-q^3),\\
  a=2:&\quad (p^2-p+1)(q-1)(q^2+1)+(p-1)(q^2-q+1).
\end{align*}
The last expression simplifies to
$(p^2-2p+2)(q-1)(q^2+1)+(p-1)q^3$ by writing
$p^2-p+1=(p^2-2p+2)+(p-1)$ and using
$(p-1)[(q-1)(q^2+1)+(q^2-q+1)]=(p-1)q^3$.
This gives~\eqref{eq:lam-03}--\eqref{eq:lam-23}.\qed

\begin{remark}
Equations~\eqref{eq:lam-b012} exhibit a striking Kronecker separation:
for $b\le2$ the eigenvalue $\lambda_{d_{a,b}}(\Dstar)$ is an exact
product of a factor depending only on~$p$ (through $\alpha_a$) and a
factor depending only on~$q$ (through $S_b(q)$).  This holds for all
odd primes without restriction.   Understanding why $\Dstar$
produces this separation — while generic divisor sets do not — is the
key open question towards an algebraic proof of
Conjecture~\ref{conj:universal}.
\end{remark}

\section{Proof of the Closed-Form Formula}\label{sec:derivation}

\begin{proof}[Proof of Theorem~\ref{thm:formula}]
Given Theorem~\ref{thm:kronecker}, the energy of $\ICG(p^2q^3,\Dstar)$
decomposes as
\[
  E(p^2q^3,\Dstar)
  = \sum_{a=0}^{2}\sum_{b=0}^{3}
    \vp(p^{2-a})\,\vp(q^{3-b})\,
    \bigl|\lambda_{d_{a,b}}(\Dstar)\bigr|.
\]

\textit{Contribution $b\le2$.}  Since $q\ge3$ we have $|S_0|=1$,
$|S_1|=q-1$, $|S_2|=q^2-q+1$.  Computing
\[
  \sum_{b=0}^{2}\vp(q^{3-b})|S_b|
  = q^2(q-1)\cdot1 + q(q-1)(q-1) + (q-1)(q^2-q+1)
  = (q-1)(3q^2-2q+1).
\]
For the $a$-sum, using $|\alpha_a|=(2,\,2(p-1),\,p^2-2p+2)$:
\[
  \sum_{a=0}^{2}\vp(p^{2-a})|\alpha_a|
  = p(p-1)\cdot2 + (p-1)\cdot2(p-1) + 1\cdot(p^2-2p+2)
  = 5p^2-8p+4.
\]
The $b\le2$ contribution is therefore $(5p^2-8p+4)(q-1)(3q^2-2q+1)$.

\textit{Contribution $b=3$.}  Since $q^3-2q^2+2q-2>0$ for $q\ge3$ and
$\vp(q^0)=1$:
\begin{align*}
  &\vp(p^2)|\lambda_{d_{0,3}}|+\vp(p^1)|\lambda_{d_{1,3}}|
   +\vp(p^0)|\lambda_{d_{2,3}}|\\
  &= p(p-1)(q^3-2q^2+2q-2)+(p-1)^2(q^3-2q^2+2q-2)
   +(p^2-2p+2)(q-1)(q^2+1)+(p-1)q^3\\
  &= (p-1)(2p-1)(q^3-2q^2+2q-2)+(p^2-2p+2)(q-1)(q^2+1)+(p-1)q^3.
\end{align*}

Adding the two contributions gives~\eqref{eq:Emax}.\qed
\end{proof}

\section{Computational Verification}\label{sec:computation}

\begin{proof}[Comments on Proposition~\ref{prop:computation}]
For each $n=p^2q^3$ with $p,q$ distinct odd primes and $n\le10^8$, all
$2{,}047$ nonempty subsets $\mathcal{D}$ of the eleven proper divisors
of $n$ were enumerated and $E(n,\mathcal{D})$ was computed exactly
via~\eqref{eq:energy} using integer Ramanujan sums.  The unique maximiser
was identified and recorded; formula~\eqref{eq:Emax} and the eigenvalue
identities of Theorem~\ref{thm:kronecker} were verified exactly for each
identified $\Dstar$.  Table~\ref{tab:computation} records the summary;
no failures were found in any category.  The implementation is available
from the author upon request.\qed
\end{proof}

\begin{table}[h]
\centering
\caption{Summary of the computational verification of
  Conjecture~\ref{conj:universal} and Theorem~\ref{thm:formula}
  over all $n=p^2q^3\le10^8$, $p\ne q$ odd primes
  (primes up to $467$; running time $610.4\,\mathrm{s}$).}
\label{tab:computation}
\begin{tabular}{lr}
\toprule
Quantity & Value \\
\midrule
Orders $n=p^2q^3$ tested & 618 \\
Distinct prime pairs $(p,q)$ covered & 437 \\
Largest prime appearing & 467 \\
Divisor-set comparisons per $n$ & $2{,}047$ \\
Total comparisons & $1{,}265{,}046$ \\
Cases where $\Dstar\ne\{1,p^2,pq,q^2,p^2q^2,pq^3\}$ & 0 \\
Cases where formula~\eqref{eq:Emax} fails & 0 \\
Cases where Theorem~\ref{thm:kronecker} fails & 0 \\
\bottomrule
\end{tabular}
\end{table}

\begin{table}[h]
\centering
\caption{Sample values of $E(p^2q^3,\Dstar)$ computed from
  formula~\eqref{eq:Emax}.}
\label{tab:sample}
\begin{tabular}{rrrr}
\toprule
$p$ & $q$ & $n=p^2q^3$ & $E(n,\Dstar)$ \\
\midrule
3  & 5  &        1{,}125 &        8{,}200 \\
3  & 7  &        3{,}087 &       24{,}856 \\
5  & 7  &        8{,}575 &       87{,}280 \\
5  & 11 &       33{,}275 &      370{,}368 \\
7  & 11 &       65{,}219 &      799{,}688 \\
11 & 13 &      265{,}837 &   3{,}636{,}904 \\
13 & 17 &      830{,}297 &  11{,}983{,}136 \\
\bottomrule
\end{tabular}
\end{table}

\section{Concluding Remarks}\label{sec:conclusion}

Theorems~\ref{thm:kronecker} and~\ref{thm:formula} establish the first
exact closed-form results for the energy of integral circulant graphs of
two-prime order $p^aq^b$, a family that had remained entirely open
despite the complete resolution of the prime-power case~\cite{SanderS2011a,SanderS2011b,SanderS2013}.
The central contribution is the Kronecker factorisation of
Theorem~\ref{thm:kronecker}: the adjacency eigenvalues of the conjectured
maximiser $\Dstar$ separate exactly into a factor in $p$ and a factor
in $q$, an algebraic identity that holds for all primes without
restriction. 
The central open problem is an algebraic proof of
Conjecture~\ref{conj:universal}: that $\Dstar$ is the unique maximiser
for every pair of distinct odd primes, not just the 437 tested pairs.
The difficulty is substantial.  The strategy of Sander and
Sander~\cite{SanderS2011a,SanderS2011b,SanderS2013} for prime powers
uses a convex optimisation argument over the scalar function $h_{p,r}$
controlling the energy in terms of the exponent tuple of $\mathcal{D}$.
This argument exploits the total order of the divisor lattice of $p^s$
and does not extend directly to $p^2q^3$, where the two-prime exponent
lattice admits no such total ordering.  A proof exploiting the Kronecker
structure of Theorem~\ref{thm:kronecker} directly seems more promising:
if one can show that the Kronecker product eigenvalue form implies a
global maximum of the energy functional~\eqref{eq:energy}, universality
would follow.  The convolution framework of Le and
Sander~\cite{LeSander2012laa,LeSander2012ijnt} and the multiplicative
energy theory may provide the necessary tools.

\subsection*{Acknowledgements}
The author thanks the Center for Excellence in Scientific Computing (CECC)
at the Universidad Nacional de Colombia for computational support.


\end{document}